\documentclass[11pt,reqno]{amsart}
\UseRawInputEncoding

\usepackage{latexsym,bm, amsfonts,amsthm,amsmath,mathrsfs}

\makeatletter
\newcommand{\rmnum}[1]{\romannumeral #1}
\newcommand{\Rmnum}[1]{\expandafter\@slowromancap\romannumeral #1@}
\makeatother

 \newtheorem{lem}{Lemma}[section]  \newtheorem{thm}{Theorem}[section]
\newtheorem{cor}{Corollary}[section] \newtheorem{defn}{Definition}[section] \newtheorem{rmk}{Remark}[section] 
\numberwithin{equation}{section}

 \newcommand{\me}{\mathrm{e}} 
\newcommand{\dif}{\mathrm{d}} \DeclareMathAlphabet{\mathsfsl}{OT1}{cmss}{m}{sl} \DeclareMathAlphabet{\mathpzc}{OT1}{pzc}{m}{it}

    \newcommand{\ee}{\mathbb{E}}

\newcommand{\nn}{\mathbb{N}}
\newcommand{\rr}{\mathbb{R}}
    
\newcommand{\vv}{\mathbb{V}}
 \def\CC{\mathcal C}   \def\FF{\mathcal F}  \def\HH{\mathcal H}

\def\d"{^{\prime\prime}} \def\d'{^{\prime}}

\begin{document}
%\begin{CJK*}{GBK}{kai}
\title[]{Equivalent conditions of complete $p$-th moment convergence for weighted sums of i. i. d. random variables under sublinear expectations}\thanks{Supported by Doctoral Scientific Research Starting Foundation of Jingdezhen Ceramic University (Nos.102/01003002031), Scientific Program of Department of Education of Jiangxi Province of China (Nos. GJJ190732, GJJ180737) and Natural Science Foundation Program of Jiangxi Province 20202BABL211005.}
\date{} \maketitle

%\subjclass[2000]{60F15, 60F05}

%\keywords{Complete moment convergence; Capacity; I. i. d. random variables; Weighted sums; Sublinear expectation}
%\date{} \maketitle

\begin{center}
 Mingzhou  Xu~\footnote{Email: mingzhouxu@whu.edu.cn} \quad   Kun Cheng\footnote{Email: chengkun0010@126.com} \\
 School of Information Engineering, Jingdezhen Ceramic University,
  Jingdezhen 333403, P. R. China
\end{center}

 \renewcommand{\abstractname}{~}

{\bf Abstract}
We investigate the complete $p$-th moment convergence for weighted sums of independent, identically distributed random variables under sublinear expectations space. Using moment inequality and truncation methods, we prove the equivalent conditions of complete $p$-th moment convergence of weighted sums of independent, identically distributed random variables under sublinear expectations space, which complement the corresponding results obtained in Guo and Shan (2020).


\vspace{-3mm}

\section{Introduction }
Peng \cite{Peng2007,Peng2010} presented the concept of  the sublinear expectation space to study the uncertainty of probability and distribution. The seminal work of Peng \cite{Peng2007,Peng2010} encourages people to study limit theorems under sublinear expectations space. Zhang \cite{Zhang2015,Zhang2016a,Zhang2016b} proved results including exponential inequalities, Rosenthal's inequalities, and Donsker's invariance principle under sublinear expectations. Wu \cite{Wu2020} obtained precise asymptotics for complete integral convergence.  Xu and Cheng \cite{Xu2021a} studied precise asymptotics in the law of iterated logarithm under sublinear expectations.  The interested reader could refer to Xu and Zhang \cite{Xujiapan2019,Xujiapan2020}, Chen \cite{Chen2016}, Gao and Xu \cite{Gao2011}, Fang et al. \cite{Fang12018},  Hu et al. \cite{Hufeng2014},  Hu and Yang \cite{Huzechun2017}, Huang and Wu \cite{Huang2019}, Kuczmaszewska \cite{Kuczm},  Ma and Wu \cite{Maxiaochen2020}, Wang and Wu \cite{Wangwenjuan2019}, Wu and Jiang \cite{Wuqunying2018}, Yu and Wu \cite{Yudonglin2018}, Zhang \cite{Zhang2016c}, Zhong and Wu \cite{Zhong2017} and references therein for more limit theorems under sublinear expectations.

Recently Guo and Shan \cite{Guo2020} studied equivalent conditions of complete $q$-th moment convergence for weighted sums of sequences of negatively orthant dependent random variables. Xu and Cheng \cite{Xu2021b} obtained equivalent conditions of complete convergence for weighted sums of sequences of i. i. d. random variables under sublinear expectations. Motivated by the work of Guo and Shan \cite{Guo2020},  Xu and Cheng \cite{Xu2021b}, here we try to prove the equivalent conditions of complete $p$-th moment convergence of weighted sums of independent, identically distributed random variables under sublinear expectations space, which complement the corresponding results obtained in Guo and Shan \cite{Guo2020}, also extend results in Xu and Cheng \cite{Xu2021b} from complete convergence to complete ~$p$-th monent convergence.

 We organized the rest of this paper as follows. In the next section, we give necessary basic notions, concepts and relevant properties, and present necessary lemmas under sublinear expectations. In Section 3, we give our main results, Theorems \ref{thm1}-\ref{thm4},  whose proofs are presented in Section 4.

\section{Preliminaries}
 As in Xu and Cheng \cite{Xu2021b}, we adopt similar notations as in the work by Peng \cite{Peng2010} and Chen \cite{Chen2016}. Suppose that $(\Omega,\FF)$ is a given measurable space. We assume that $\HH$ is a subset of all random variables on $(\Omega,\FF)$ such that $I_A\in \HH$ (cf. Chen \cite{Chen2016}), where $A\in\FF$, and  $X_1,\cdots,X_n\in \HH$ implies $\varphi(X_1,\cdots,X_n)\in \HH$ for each $\varphi\in \CC_{l,Lip}(\rr^n)$, where $\CC_{l,Lip}(\rr^n)$ denotes the linear space of (local lipschitz) function $\varphi$ satisfying
$$
|\varphi(\mathbf{x})-\varphi(\mathbf{y})|\le C(1+|\mathbf{x}|^m+|\mathbf{y}|^m)(|\mathbf{x}-\mathbf{y}|), \forall \mathbf{x},\mathbf{y}\in \rr^n
$$
for some $C>0$, $m\in \nn$ depending on $\varphi$.
\begin{defn}\label{defn1} A sublinear expectation $\ee$ on $\HH$ is a functional $\ee:\HH\mapsto \bar{\rr}:=[-\infty,\infty]$ satisfying the following properties: for all $X,Y\in \HH$, we have
\begin{description}
\item[\rm (a)]  Monotonicity: If $X\ge Y$, then $\ee[X]\ge \ee[Y]$;
\item[\rm (b)] Constant preserving: $\ee[c]=c$, $\forall c\in\rr$;
\item[\rm (c)] Positive homogeneity: $\ee[\lambda X]=\lambda\ee[X]$, $\forall \lambda\ge 0$;
\item[\rm (d)] Sub-additivity: $\ee[X+Y]\le \ee[X]+\ee[Y]$ whenever $\ee[X]+\ee[Y]$ is not of the form $\infty-\infty$ or $-\infty+\infty$.
 \end{description}

\end{defn}

A set function $V:\FF\mapsto[0,1]$ is said to be a capacity if it obeys
\begin{description}
\item[\rm (a)]$V(\emptyset)=0$, $V(\Omega)=1$;
\item[\rm (b)]$V(A)\le V(B)$, $A\subset B$, $A,B\in \FF$.\\
Moreover, if $V$ is continuous, then $V$ should satisfy
\item[\rm (c)] $V(A_n)\uparrow V(A)$, if $A_n\uparrow A$.
\item[\rm (d)] $V(A_n)\downarrow V(A)$, if $A_n\downarrow A$.
 \end{description}
A capacity $V$ is called sub-additive if $V(A+B)\le V(A)+V(B)$, $A,B\in \FF$.

In this paper, given a sublinear expectation space $(\Omega, \HH, \ee)$, set $\vv(A):=\inf\{\ee[\xi]:I_A\le \xi, \xi\in \HH\}=\ee[I_A]$, $\forall A\in \FF$ (see (2.3) and the definitions of $\vv$ above (2.3) in Zhang \cite{Zhang2016a}). Clearly $\vv$ is a sub-additive capacity. Denote the Choquet expectations $C_{\vv}$ by
$$
C_{\vv}(X):=\int_{0}^{\infty}\vv(X>x)\dif x +\int_{-\infty}^{0}(\vv(X>x)-1)\dif x.
$$

Suppose that $\mathbf{X}=(X_1,\cdots, X_m)$, $X_i\in\HH$ and $\mathbf{Y}=(Y_1,\cdots,Y_n)$, $Y_i\in \HH$  are two random vectors on  $(\Omega, \HH, \ee)$. $\mathbf{Y}$ is called to be independent of $\mathbf{X}$, if for each Borel-measurable function $\psi$ on $\rr^m\times \rr^n$ with $\psi(\mathbf{X},\mathbf{Y}), \psi(\mathbf{x},\mathbf{Y})\in \HH$ for each $\mathbf{x}\in\rr^m$, we have $\ee[\psi(\mathbf{X},\mathbf{Y})]=\ee[\ee\psi(\mathbf{x},\mathbf{Y})|_{\mathbf{x}=\mathbf{X}}]$ whenever $\bar{\psi}(\mathbf{x}):=\ee[|\psi(\mathbf{x},\mathbf{Y})|]<\infty$ for each $\mathbf{x}$ and $\ee[|\bar{\psi}(\mathbf{X})|]<\infty$ (see Definition 2.5 in Chen \cite{Chen2016} ). $\{X_n\}_{n=1}^{\infty}$ is called a sequence of independent random variables, if $X_{n+1}$ is independent of $(X_1,\cdots,X_n)$ for each $n\ge 1$.

Assume that $\mathbf{X}_1$ and $\mathbf{X}_2$ are two $n$-dimensional random vectors defined, respectively, in sublinear expectation spaces $(\Omega_1,\HH_1,\ee_1)$ and $(\Omega_2,\HH_2,\ee_2)$. They are called identically distributed if  for every Borel-measurable function $\psi$ such that $\psi(\mathbf{X}_1)\in \HH_1, \psi(\mathbf{X}_2)\in \HH_2$,
$$
\ee_1[\psi(\mathbf{X}_1)]=\ee_2[\psi(\mathbf{X}_2)], \mbox{  }
$$
whenever the sublinear expectations are finite. $\{X_n\}_{n=1}^{\infty}$ is called to be identically distributed if for each $i\ge 1$, $X_i$ and $X_1$ are identically distributed.

In the sequel we suppose that $\ee$ is countably sub-additive, i.e., $\ee(X)\le \sum_{n=1}^{\infty}\ee(X_n)$, whenever $X\le \sum_{n=1}^{\infty}X_n$, $X,X_n\in \HH$, and $X\ge 0$, $X_n\ge 0$, $n=1,2,\ldots$. Let $C$ denote a positive constant which may differ from line to line. $I(A)$ or $I_A$ represents the indicator function of $A$, $a_n\ll b_n$ means that there exists a constant $C>0$ such that $a_n\le C b_n$ for $n$ large sufficiently and $a_n\approx b_n$ means that $a_n\ll b_n$ and $b_n\ll a_n$. We use $\log x$ for $\ln\max\{\me, x\}$.

We first present several necessary lemmas to prove our main results. By using  Corollary 2.2, Theorem 2.3 in Zhang \cite{Zhang2016b}, the proofs of Theorem 3 in M\'{o}ricz \cite{Moricz1976} and Minkowski's inequality under sublinear expectations, we see that the following lemma holds.
\begin{lem}\label{lem1}[also cf. Lemma 2.3 in Xu and Cheng \cite{Xu2021b}]Suppose that $\{X_i,1\le i\le n\}$ is a sequence of independent random variables under sublinear expectation space $(\Omega,\HH,\ee)$ with $\ee[X_i]\le 0$, $\ee|X_i|^M<\infty$, $1\le i\le n$, $M\ge 2.$
Then
\begin{eqnarray*}
\ee\max_{1\le j\le n}\left|\sum_{i=1}^{j}X_i\right|^M\le C\log^M n\left(\sum_{i=1}^{n}\ee|X_i|^M+\left(\sum_{i=1}^{n}\ee|X_i|^2\right)^{M/2}\right),
 \end{eqnarray*}
 where $C$ depends on $M$ only.
\end{lem}
\begin{proof}For readers' convenience, we give complete proofs here. First for $b\ge 0$, $n\ge 1$, set $S_{b,n}=\sum_{k=b+1}^{b+n}X_i$, $L_{b,n}=\max_{1\le k\le n}|S_{b,k}|$, $$g(F_{b,n})=C_M\left(\sum_{i=b+1}^{b+n}\ee|X_i|^M+\left(\sum_{i=b+1}^{b+n}\ee|X_i|^2\right)^{M/2}\right),$$
where $C_M$ depends on $M$ determined as in (2.8) of Corollary 2.2 (b)  in Zhang \cite{Zhang2016b}.
By Corollary 2.2 in Zhang \cite{Zhang2016b}, we know that for all $b\ge 0$, $n\ge 1$,
\begin{equation}\label{2.01}
\ee(|S_{b,n}|^M)\le g(F_{b,n}).
\end{equation}
Obviously, for $b\ge 0$, $1\le k\le k+l$,
\begin{equation}\label{2.02}
g(F_{b,k})+g(F_{b+k,l})\le g(F_{b,k+l}).
\end{equation}
Set $\Lambda(1)=1$, and for $n\ge 2$, $\Lambda(n)=1+\Lambda(m-1)$, where $m$ is the integer part of $\frac12(n+2)$. From $1+\log(2(m-1))\le \log(2n)$ follows that $\Lambda(n)\le \log(2n)$.
Now it is enough to prove that for $b\ge 0$, $n\ge 1$,
\begin{equation}\label{2.03}
\ee(L_{b,n}^M)\le (\Lambda(n))^Mg(F_{b,n}).
\end{equation}
 As in the proof of Theorem 4 of M\'{o}ricz \cite{Moricz1976}, Let $n>1$ be given. Obviously, $n=2m-1$ or $2m-2$. For $b\ge 0$, $m\le k\le n$, we see that
$$
|S_{b,k}|\le |S_{b,m}|+|S_{b+m,k-m}|,
$$
whence, for such $k$'s,
$$
|S_{b,k}|\le |S_{b,m}|+L_{b+m,n-m}.
$$
Since for $1\le k< m$, we have $|S_{b,k}|\le L_{b,m-1}$, hence, for $1\le k\le n$,
$$
|S_{b,k}|\le |S_{b,m}|+\left(L_{b,m-1}^M+L_{b+m,n-m}^M\right)^{1/M}.
$$
Thus
$$
L_{b,n}\le |S_{b,m}|+\left(L_{b,m-1}^M+L_{b+m,n-m}^M\right)^{1/M},
$$
and, by Minkowski's inequality under sublinear expectations (see (4.10) in  Proposition 4.2 of Chapter {\Rmnum1} of Peng \cite{Peng2010}),
\begin{equation}\label{2.04}
\left(\ee(L_{b,n}^M)\right)^{1/M}\le \left(\ee(|S_{b,m}^M|)\right)^{1/M}+\left(\ee(L_{b,m-1}^M)+\ee(L_{b+m,n-m}^M)\right)^{1/M}
\end{equation}
Suppose now that (\ref{2.03}) holds for $k<n$. Then by the choice of $m$, we see that
$$
\ee(L_{b,m-1}^M)\le \Lambda^M(m-1)g(F_{b,m-1}),
$$
and
$$
\ee(L_{b+m,n-m}^M)\le \Lambda^{M}(n-m)g(F_{b+m,n-m})\le \Lambda^{M}(m-1)g(F_{b+m,n-m}).
$$
By these two inequalities above and (\ref{2.02}), we see that
\begin{equation}\label{2.05}
\ee(L_{b,m-1}^M)+\ee(L_{b+m,n-m}^M)\le \Lambda^{M}(m-1)g(F_{b,m}).
\end{equation}
Finally (\ref{2.01}) implies
\begin{equation}\label{2.06}
\ee(|S_{b,m}|^M)\le g(F_{b,m})\le g(F_{b,n}).
\end{equation}
By (\ref{2.04})-(\ref{2.06}), we conclude that
$$
\left(\ee(L_{b,n}^M)\right)^{1/M}\le \left(1+\Lambda(m-1)\right)g^{1/M}(F_{b,n})=\Lambda(n)g^{1/M}(F_{b,n}),
$$
which implies the result. Hence, by (\ref{2.01}), the conclusion of (\ref{2.03}) is true for $n=1$. By induction (\ref{2.03}) holds for all $n=1,2,\ldots$. The proof is complete.

\end{proof}
\begin{lem}\label{lem2}[ See Lemma 2.4 in Xu and Cheng \cite{Xu2021b,Xu2021c}]Let $\{X_n;n\ge 1\}$ be a sequence of independent random variables under sublinear expectation space $(\Omega,\HH,\ee)$. Then for all $n\ge 1$ and $x>0$,
Then
\begin{eqnarray*}
&&\left[1-\vv\left(\max_{1\le j\le n}|X_j|>x\right)\right]^2\sum_{j=1}^{n}\vv(|X_j|>x)\le 4\vv\left(\max_{1\le j\le n}|X_j|>x\right).
 \end{eqnarray*}
\end{lem}
\begin{rmk}
In the proofs of Lemma 2.4 in Xu and Cheng \cite{Xu2021b,Xu2021c}, by independence of $I(A_k):=I(|X_k|>x), k=1,\ldots, n$, $\ee(X+c)=\ee(c+X)=\ee(X)+c$ for constant $c$ and $X\in \HH$, Definition 2.5 in Chen \cite{Chen2016}, we see that
$$
 \begin{aligned}
&\quad\sum_{k=1}^{n}\ee[I(A_k)]\mbox{$=\sum_{k=1}^{n-2}\ee[I(A_k)]+\ee\left[I(A_{n-1})+\ee\left[I(A_n)\right]\right]$}\\
&=\sum_{k=1}^{n-2}\ee[I(A_k)]+\ee\left[\left[x+\ee\left[I(A_n)\right]\right]|_{x=I(A_{n-1})}\right]=\sum_{k=1}^{n-2}\ee[I(A_k)]+\ee[\ee[x+I(A_n)]|_{x=I(A_{n-1})}]\\
&=\sum_{k=1}^{n-2}\ee[I(A_k)]+\ee\left[I(A_{n-1})+I(A_n)\right]=\cdots=\ee\left[I(A_1)+\ee\left[\sum_{k=2}^{n}I(A_k)\right]\right]\\
&=\ee\left[\sum_{k=1}^{n}I(A_k)\right],
 \end{aligned}
$$
which implies that Lemma \ref{lem2} is valid. 
\end{rmk}
\begin{lem}\label{lem3}Let $X$ be a random variable under sublinear expectation space $(\Omega,\HH,\ee)$, $q>0$, $r>0$ and $p>0$. Then the following are equivalent:
\begin{description}
 \item[\rm (i)]\begin{eqnarray*}
\begin{cases} C_{\vv}\left(|X|^{p}\right)<\infty,& \text{
for $p>r/q$,}\\
C_{\vv}\left(|X|^{r/q}\ln(|X|)\right)<\infty, & \text{ for $p=r/q$,}\\
C_{\vv}\left(|X|^{r/q}\right)<\infty, & \text{ for $p<r/q$,}
\end{cases}
\end{eqnarray*}
\item[\rm (ii)]
\begin{equation*}
\int_{1}^{\infty}\dif y\int_{1}^{\infty}y^{r-1}\vv(|X|>x^{1/p}y^q)\dif x<\infty.
\end{equation*}
\end{description}
\end{lem}
\begin{proof}
\begin{eqnarray*}
&&\int_{1}^{\infty}\dif y\int_{1}^{\infty}y^{r-1}\vv(|X|>x^{1/p}y^q)\dif x=\int_{1}^{\infty}\dif t\int_{t^q}^{\infty}t^{r-1}\vv(|X|>s)s^{p-1}t^{-qp}\dif s\\
&&\mbox{ ( Setting $s=x^{1/p}y^q$, $t=y$)}\\
&&\approx
\begin{cases} \int_{1}^{\infty}\vv(|X|>s)s^{p-1}\dif s\approx \int_{1}^{\infty}\vv(|X|>s)s^{p-1}\dif s\approx C_{\vv}\left(|X|^{p}\right),& \text{
for $p>r/q$,}\\
\int_{1}^{\infty}\vv(|X|>s)s^{r/q-1}\ln(s) \dif s\approx C_{\vv}\left(|X|^{r/q}\ln(|X|)\right), & \text{ for $p=r/q$,}\\
\int_{1}^{\infty}\vv(|X|>s)s^{p-1+\frac{r-pq}{q}} \dif s\approx C_{\vv}\left(|X|^{r/q}\right), & \text{ for $p<r/q$.}
\end{cases}
\end{eqnarray*}
The proof is finished.
\end{proof}
\begin{lem}\label{lem4}Let $X$ be a random variable under sublinear expectation space $(\Omega,\HH,\ee)$, $q>0$, $r>0$ and $p>0$. Then the following is equivalent:
\begin{description}
 \item[\rm (i)]\begin{eqnarray*}
\begin{cases} C_{\vv}\left(|X|^{p}\right)<\infty,& \text{
for $p>r/q$,}\\
C_{\vv}\left(|X|^{r/q}\ln^2|X|\right)<\infty, & \text{ for $p=r/q$,}\\
C_{\vv}\left(|X|^{r/q}\ln|X|\right)<\infty, & \text{ for $p<r/q$,}
\end{cases}
\end{eqnarray*}
\item[\rm (ii)]
\begin{equation*}
\int_{1}^{\infty}\dif y\int_{1}^{\infty}y^{r-1}\ln(y)\vv(|X|>x^{1/p}y^q)\dif x<\infty.
\end{equation*}
\end{description}
\end{lem}
\begin{proof}
\begin{eqnarray*}
&&\int_{1}^{\infty}\dif y\int_{1}^{\infty}y^{r-1}\ln(y)\vv(|X|>x^{1/p}y^q)\dif x\approx\int_{1}^{\infty}\dif t\int_{t^q}^{\infty}t^{r-1}\ln(t)\vv(|X|>s)s^{p-1}t^{-qp}\dif s\\
&&\mbox{ ( Setting $s=x^{1/p}y^q$, $t=y$)}\\
&&\approx \int_{1}^{\infty}\vv(|X|>s)s^{p-1}\dif s\int_{1}^{s^{1/q}}t^{r-1-qp}\ln(t)\dif t\\
&&\approx
\begin{cases} \int_{1}^{\infty}\vv(|X|>s)s^{p-1}\dif s\approx C_{\vv}\left(|X|^{p}\right),& \text{
for $p>r/q$,}\\
\int_{1}^{\infty}\vv(|X|>s)s^{r/q-1}\ln^2(s) \dif s\approx C_{\vv}\left(|X|^{r/q}\ln^2|X|\right), & \text{ for $p=r/q$,}\\
\int_{1}^{\infty}\vv(|X|>s)s^{p-1+\frac{r-pq}{q}}\ln(s) \dif s\approx C_{\vv}\left(|X|^{r/q}\ln|X|\right), & \text{ for $p<r/q$.}
\end{cases}
\end{eqnarray*}
This finishes the proof.
\end{proof}
\section{Main results}
We state our main results, the proofs of which will be given in Section 4.
\begin{thm}\label{thm1} Let $\{X_n,n\ge 1\}$ be a sequence of independent random variables, identically distributed as $X$ under sublinear expectation space  $(\Omega,\HH,\ee)$. Assume that $r>1$, $q>\frac12$, $\beta>-q/r$, and suppose that $\ee X=-\ee(-X)=0$ for $\frac12<q\le 1$. Suppose that $\{a_{ni}\approx (i/n)^{\beta}(1/n^q), 1\le i\le n, n\ge 1\}$ is a triangular array of real numbers. Then the following is equivalent:
\begin{description}
 \item[\rm (i)]\begin{eqnarray}\label{3.1}
\begin{cases} C_{\vv}\left(|X|^{p}\right)<\infty,& \text{
for $p>r/q$,}\\
C_{\vv}\left(|X|^{r/q}\ln|X|\right)<\infty, & \text{ for $p=r/q$,}\\
C_{\vv}\left(|X|^{r/q}\right)<\infty, & \text{ for $p<r/q$,}
\end{cases}
\end{eqnarray}
\item[\rm (ii)]
\begin{equation}\label{3.2}
\sum_{n=1}^{\infty}n^{r-2}C_{\vv}\left(\left(\max_{1\le k\le n}\left|\sum_{i=1}^{k}a_{ni}X_{i}\right|^p-\epsilon\right)^{+}\right)<\infty, \mbox{ $\forall \epsilon>0$.}
\end{equation}
\end{description}
\end{thm}
\begin{rmk}
If $(\Omega,\HH,\ee)$ is classic probability space, then Theorem \ref{thm1} recovers Theorem 7 in Guo and Shan \cite{Guo2020} in case that $\{X_n,n\ge 1\}$ is a sequence of independent random variables, identically distributed as $X$. As pointed in  Hossein and Nezakati \cite{Hosseni2019}, why we need $p$-th moment convergence under sublinear expectations, the complete moment covergence is a more general expression than complete convergence in theory and practice, the interested reader also could refer to the first study in complete moment convergence by Chow \cite{Chow1988}.
\end{rmk}
\begin{rmk}\label{remark01}Under the same conditions in Theorem \ref{thm1} , if (\ref{3.2}) holds, by the same proof of Remark 3.2 of Hossein and Nezakati \cite{Hosseni2019}, we see that for all $\epsilon>0$
\begin{equation}\label{3.201}
\sum_{n=1}^{\infty}n^{r-2}\vv\left(\max_{1\le k\le n}\left|\sum_{i=1}^{k}a_{ni}X_{i}\right|^p>\epsilon\right)<\infty.
\end{equation}
Moreover if $\vv$ is continuous, from (\ref{3.201}) follows that
\begin{equation}\label{3.202}
\sum_{i=1}^{n}a_{ni}X_{i}\rightarrow 0 \mbox{  a.s. $\vv$, as $n\rightarrow \infty$.}
\end{equation}
In (\ref{3.202}), if $1/2<q\le1$, $\beta=0$, we conclude that $\sum_{i=1}^{n}X_{i}/n^q\rightarrow 0 \mbox{  a.s. $\vv$, as $n\rightarrow \infty$,}$ which is similar to Theorem 1 of Zhang and Lin \cite{Zhang2018}.
\end{rmk}
\begin{thm}\label{thm2} Let $\{X_n,n\ge 1\}$ be a sequence of independent random variables, identically distributed as $X$ under sublinear expectation space  $(\Omega,\HH,\ee)$. Assume that $r>1$, $q>\frac12$, $\beta=-q/r<0$, and suppose that $\ee X=-\ee(-X)=0$ for $\frac12<q\le 1$. Suppose that $\{a_{ni}\approx (i/n)^{\beta}(1/n^q), 1\le i\le n, n\ge 1\}$ is a triangular array of real numbers. Then (\ref{3.2}) equivalent to
\begin{eqnarray}\label{3.3}
\begin{cases} C_{\vv}\left(|X|^{p}\right)<\infty,& \text{
for $p>r/q$,}\\
C_{\vv}\left(|X|^{r/q}\ln^2|X|\right)<\infty, & \text{ for $p=r/q$,}\\
C_{\vv}\left(|X|^{r/q}\ln|X|\right)<\infty, & \text{ for $p<r/q$.}
\end{cases}
\end{eqnarray}
\end{thm}

\begin{thm}\label{thm3} Let $\{X_n,n\ge 1\}$ be a sequence of independent random variables, identically distributed as $X$ under sublinear expectation space  $(\Omega,\HH,\ee)$. Assume that $r>1$, $q>\frac12$, $-q<\beta<-q/r<0$, and suppose that $\ee X=-\ee(-X)=0$ for $\frac12<q\le 1$. Suppose that $\{a_{ni}\approx (i/n)^{\beta}(1/n^q), 1\le i\le n, n\ge 1\}$ is a triangular array of real numbers. Then (\ref{3.2}) equivalent to
\begin{eqnarray}\label{3.4}
\begin{cases} C_{\vv}\left(|X|^{p}\right)<\infty,& \text{
for $p>(r-1)/(q+\beta)$,}\\
C_{\vv}\left(|X|^{(r-1)/(q+\beta)}\ln|X|\right)<\infty, & \text{ for $p=(r-1)/(q+\beta)$,}\\
C_{\vv}\left(|X|^{(r-1)/(q+\beta)}\right)<\infty, & \text{ for $p<(r-1)/(q+\beta)$.}
\end{cases}
\end{eqnarray}
\end{thm}
As in the proofs of Theorems \ref{thm1},\ref{thm2} and \ref{thm3}, we can get the following corollary.
\begin{cor}\label{cor} Let $\{X_n,n\ge 1\}$ be a sequence of independent random variables, identically distributed as $X$ under sublinear expectation space  $(\Omega,\HH,\ee)$. Assume that $r>1$, $q>\frac12$, $\beta>-q$, and suppose that $\ee X=-\ee(-X)=0$ for $\frac12<q\le 1$. Suppose that $\{a_{ni}\approx ((n-i)/n)^{\beta}(1/n^q), 0\le i\le n-1, n\ge 1\}$ is a triangular array of real numbers. Then
\begin{description}
 \item[\rm (i)] (\ref{3.1}) is equivalent to
 \begin{equation}\label{3.5}
\sum_{n=1}^{\infty}n^{r-2}\ee\left(\left(\max_{0\le k\le n-1}\left|\sum_{i=0}^{k}a_{ni}X_{i}\right|^p-\epsilon\right)^{+}\right)<\infty, \mbox{ $\beta>-q/r$.}
\end{equation}
\item[\rm (ii)] (\ref{3.3}) is equivalent to (\ref{3.5}) when $\beta=-q/r$.
\item[\rm (iii)] (\ref{3.4}) is equivalent to (\ref{3.5}) when $-q<\beta<-q/r$.
\end{description}
\end{cor}
The following theorem is complete $p$-th moment convergence on Ces\`{a}ro summation of independent, identically distributed random variables
under$(\Omega,\HH,\ee)$.
\begin{thm}\label{thm4} Let $\{X_n,n\ge 1\}$ be a sequence of independent random variables, identically distributed as $X$ under sublinear expectation space  $(\Omega,\HH,\ee)$. Assume that $r>1$, $q>\frac12$, $0<\alpha\le 1$, $p>0$ and suppose that $\ee X=-\ee(-X)=0$ for $\frac12<q\le 1$. Suppose $A_n^{\alpha}=[(\alpha+1)(\alpha+2)\cdots(\alpha+n)]/n!$, $n=1,2,\ldots$ and $A_0^{\alpha}=1$. Then,
\begin{description}
 \item[\rm (i)](\ref{3.1}) equivalent to
\begin{equation}\label{3.6}
\sum_{n=1}^{\infty}n^{r-2}(A_n^{\alpha})^{-p}\ee\left(\left(\max_{0\le k\le n-1}\left|\sum_{i=0}^{k}A_{n-i}^{\alpha-1}X_{i}\right|^p-\epsilon(A_n^{\alpha})^{p}\right)^{+}\right)<\infty, \mbox{ $\forall\epsilon>0$,}
\end{equation}
when $1-1/r<\alpha\le 1$.
 \item[\rm (ii)] (\ref{3.3}) is equivalent to (\ref{3.6}) when $\alpha=1-1/r$.
  \item[\rm (iii)] (\ref{3.6}) is equivalent to
  \begin{eqnarray}\label{3.7}
\begin{cases} C_{\vv}\left(|X|^{p}\right)<\infty,& \text{
for $p>(r-1)/(q\alpha)$,}\\
C_{\vv}\left(|X|^{(r-1)/(q\alpha)}\ln|X|\right)<\infty, & \text{ for $p=(r-1)/(q\alpha)$,}\\
C_{\vv}\left(|X|^{(r-1)/(q\alpha)}\right)<\infty, & \text{ for $p<(r-1)/(q\alpha)$,}
\end{cases}
\end{eqnarray}
when $0<\alpha<1-1/r$.
\end{description}
\end{thm}
\begin{rmk}
If $(\Omega,\HH,\ee)$ is classic probability space, then Theorems \ref{thm1}, \ref{thm2}, \ref{thm3}, Corollary \ref{cor} recovers respectively Theorems 10, 11, 14, Corollary 13 in Guo and Shan \cite{Guo2020} in case that $\{X_n,n\ge 1\}$ is a sequence of independent random variables, identically distributed as $X$.
\end{rmk}
\section{Proofs of the Main Results}
\begin{proof}[Proof of Theorem \ref{thm1}] We first prove that (\ref{3.1}) implies (\ref{3.2}). Notice that
\begin{eqnarray*}
&&\sum_{n=1}^{\infty}n^{r-2}C_{\vv}\left(\left(\max_{1\le k\le n}\left|\sum_{i=1}^{k}a_{ni}X_{i}\right|^p-\epsilon\right)^{+}\right)\\
&&=\sum_{n=1}^{\infty}n^{r-2}\int_{0}^{\infty}\vv\left(\max_{1\le k\le n}\left|\sum_{i=1}^{k}a_{ni}X_{i}\right|^p>\epsilon+x\right)\dif x\\
&&= \sum_{n=1}^{\infty}n^{r-2}\int_{\epsilon}^{1}\vv\left(\max_{1\le k\le n}\left|\sum_{i=1}^{k}a_{ni}X_{i}\right|^p>x\right)\dif x+\sum_{n=1}^{\infty}n^{r-2}\int_{1}^{\infty}\vv\left(\max_{1\le k\le n}\left|\sum_{i=1}^{k}a_{ni}X_{i}\right|^p>x\right)\dif x\\
&&\le\sum_{n=1}^{\infty}n^{r-2}\vv\left(\max_{1\le k\le n}\left|\sum_{i=1}^{k}a_{ni}X_{i}\right|>\epsilon^{1/p}\right)+\sum_{n=1}^{\infty}n^{r-2}\int_{1}^{\infty}\vv\left(\max_{1\le k\le n}\left|\sum_{i=1}^{k}a_{ni}X_{i}\right|>x^{1/p}\right)\dif x\\
&&:=\Rmnum{1}+\Rmnum{2}.
\end{eqnarray*}
From Xu and Cheng \cite{Xu2021b}, and the fact that (3.1) implies $C_{\vv}(|X|^{r/q})<\infty$,  we see that $\Rmnum{1}<\infty$. We next establish $\Rmnum{2}<\infty$. Choose $0<\alpha<1/p$, $\delta>0$ small sufficiently and integer $K$ large enough. For every $1\le i\le n$, $n\ge 1$, we note the fact that $n$ is large sufficiently to guarantee $x^{\alpha}n^{-\delta}<\frac{x^{1/p}}{4K}$. Since the first finite terms of the series do not affect the convergence of the series, without loss of restictions, we could assume the definitions of $X_{ni}^{(j)}$ below are meaningful. Write
\begin{eqnarray*}
&&X_{ni}^{(1)}=-x^{\alpha}n^{-\delta}I(a_{ni}X_i<-x^{\alpha}n^{-\delta})+a_{ni}X_iI(|a_{ni}X_i|\le x^{\alpha}n^{-\delta})+x^{\alpha}n^{-\delta}I(a_{ni}X_i>x^{\alpha}n^{-\delta});\\
&&X_{ni}^{(2)}=(a_{ni}X_i-x^{\alpha}n^{-\delta})I\left(x^{\alpha}n^{-\delta}<a_{ni}X_i<\frac{x^{1/p}}{4K}\right);\\
&&X_{ni}^{(3)}=(a_{ni}X_i+x^{\alpha}n^{-\delta})I\left(-\frac{x^{1/p}}{4K}<a_{ni}X_i<-x^{\alpha}n^{-\delta}\right);\\
&&X_{ni}^{(4)}=(a_{ni}X_i+x^{\alpha}n^{-\delta})I\left(a_{ni}X_i\le -\frac{x^{1/p}}{4K}\right)+(a_{ni}X_i-x^{\alpha}n^{-\delta})I\left(a_{ni}X_i\ge\frac{x^{1/p}}{4K}\right),
\end{eqnarray*}
Observe that $\sum_{i=1}^{k}a_{ni}X_i=\sum_{i=1}^{k}X_{ni}^{(1)}+\sum_{i=1}^{k}X_{ni}^{(2)}+\sum_{i=1}^{k}X_{ni}^{(3)}+\sum_{i=1}^{k}X_{ni}^{(4)}$. Notice that
\[
\left(\max_{1\le k\le n}\left|\sum_{i=1}^{k}a_{ni}X_{i}\right|>x^{1/p}\right)\subset\bigcup_{j=1}^{4}\left(\max_{1\le k\le n}\left|\sum_{i=1}^{k}X_{ni}^{(j)}\right|>x^{1/p}/4\right).
\]
Consequently, to prove (\ref{3.2}), we only need to prove that
\[
J_j:=\sum_{n=1}^{\infty}n^{r-2}\int_{1}^{\infty}\vv\left(\max_{1\le k\le n}\left|\sum_{i=1}^{k}X_{ni}^{(j)}\right|>x^{1/p}/4\right)\dif x<\infty, \mbox{  $j=1,2,3,4$.}
\]
From the definition of $X_{ni}^{(4)}$, we deduce that
\[
\left(\max_{1\le k\le n}\left|\sum_{i=1}^{k}X_{ni}^{(4)}\right|>x^{1/p}/4\right)\subset\left(\max_{1\le i\le n}|a_{ni}X_{ni}|>\frac{x^{1/p}}{4K}\right).
\]
Observe that $\beta>-q/r$ implies $\beta(r-1)/(q+\beta)>-1$. Therefore,
\begin{equation}\label{3.8}
\int_{1}^{s^{1/q}}t^{\beta(r-1)/(q+\beta)}\dif t\approx s^{\frac{1}{q}+\frac{\beta(r-1)}{q(q+\beta)}}.
\end{equation}
From $a_{ni}\approx (i/n)^{\beta}(1/n^q)$, we see that
\begin{eqnarray}\label{3.9}
\nonumber J_4&\le&\sum_{n=1}^{\infty}n^{r-2}\sum_{i=1}^{n}\int_{1}^{\infty}\vv\left(|a_{ni}X_{ni}|>\frac{x^{1/p}}{4K}\right)\dif x\\
\nonumber &\approx&\sum_{n=1}^{\infty}n^{r-2}\sum_{i=1}^{n}\int_{1}^{\infty}\vv\left(|X|>\frac{x^{1/p}}{4CK}n^{q+\beta}i^{-\beta}\right)\dif x\\
\nonumber &\approx& \int_{1}^{\infty}\dif x\int_{1}^{\infty}u^{r-2}\dif u \int_{1}^{u}\vv\left(|X|>\frac{x^{1/p}}{4CK}u^{q+\beta}v^{-\beta}\right)\dif v\\
\nonumber && \mbox{  ( Setting $s=u^{q+\beta}v^{-\beta}, t=v$)}\\
\nonumber &\approx& \int_{1}^{\infty}\dif x\int_{1}^{\infty} \dif s \int_{1}^{s^{1/q}}s^{(r-1)/(q+\beta)-1}t^{\beta(r-1)/(q+\beta)}\vv\left(|X|>\frac{x^{1/p}}{4CK}s\right)\dif t\\
&\approx&\int_{1}^{\infty}\dif x\int_{1}^{\infty}s^{r/q-1}\vv\left(|X|>\frac{x^{1/p}}{4CK}s\right)\dif s.
\end{eqnarray}
Hence, from Lemma \ref{lem3} and (\ref{3.1}), we obtain $J_4<\infty$. From the definition of $X_{ni}^{(2)}$, we conclude that $X_{ni}^{(2)}\ge0$. From the subadditivity of capacity and Definition 2.5 in Chen \cite{Chen2016} follows that
\begin{eqnarray}\label{3.10}
\nonumber &&\vv\left(\max_{1\le k\le n}\left|\sum_{i=1}^{k}X_{ni}^{(2)}\right|>x^{1/p}/4\right)=\vv\left(\sum_{i=1}^{n}X_{ni}^{(2)}>x^{1/p}/4\right)\\
\nonumber&&\le\vv\left(\mbox{ there are at least $K$ indices $i\in [1,n]$ such that $a_{ni}X_{ni}>x^{\alpha}n^{-\delta}$}\right)\\
&&\nonumber\le\sum_{1\le i_1<i_2<\cdots<i_K\le n}\vv(a_{ni_{j}}X_{i_{j}}>x^{\alpha}n^{-\delta}, \mbox{  for all $1\le j\le K$})\\
&&\nonumber\quad\mbox{(In Definition 2.5 in Chen \cite{Chen2016}, we set $X=(a_{ni_1}X_{i_1},\ldots,a_{ni_{K-1}}X_{i_{K-1}}),Y=a_{ni_K}X_{i_K}$,)}\\
&&\nonumber\quad\mbox{($\varphi(X,Y_n)=I(a_{ni_1}X_{i_1}>x^{\alpha}n^{-\delta},\ldots,a_{ni_{K-1}}X_{i_{K-1}}>x^{\alpha}n^{-\delta}, a_{ni_K}X_{i_K}>x^{\alpha}n^{-\delta})$)}\\
&&\nonumber=\sum_{1\le i_1<i_2<\cdots<i_K\le n}\ee\left[I(a_{ni_1}X_{i_1}>x^{\alpha}n^{-\delta},\ldots,a_{ni_{K-1}}X_{i_{K-1}}>x^{\alpha}n^{-\delta})\right]\vv\left(a_{ni_K}X_{i_K}>x^{\alpha}n^{-\delta}\right)\\
&&\nonumber= \ldots\ldots\\
&&=\sum_{1\le i_1<i_2<\cdots<i_K\le n}\prod_{j=1}^{K}\vv(a_{ni_{j}}X_{i_{j}}>x^{\alpha}n^{-\delta})\le \left[\sum_{j=1}^{n}\vv(a_{nj}X>x^{\alpha}n^{-\delta})\right]^K.
\end{eqnarray}
From $\beta>-q/r$, we obtain $\sum_{i=1}^{n}a_{ni}^{r/q}\approx \sum_{i=1}^{n}n^{-r(q+\beta)/q}i^{\beta r/q}\approx n^{1-r}$.
Since (\ref{3.1}) implies $\ee|X|^{r/q}<\infty$, by Markov's inequality under sublinear expectations ( cf. (9) in Hu and Wu \cite{Hu2021}) and (\ref{3.10}), we conclude that
\begin{eqnarray}\label{3.11}
\nonumber J_2 &\le&\sum_{n=1}^{\infty}n^{r-2} \int_{1}^{\infty}\left[\sum_{j=1}^{n}\vv(a_{nj}X>x^{\alpha}n^{-\delta})\right]^K \dif x\\
\nonumber&\le&\sum_{n=1}^{\infty}n^{r-2} \int_{1}^{\infty}\left(\sum_{j=1}^{n}x^{-r\alpha/q}n^{r\delta /q}a_{nj}^{r/q}\ee|X|^{r/q}\right)^K\dif x\\
&\approx&\sum_{n=1}^{\infty}n^{r-2-K(r-1-r\delta/q)}\int_{1}^{\infty}x^{-rK\alpha/q}\dif x.
\end{eqnarray}
Notice that $r>1$, $\alpha>0$, we could choose $\delta$ small sufficiently and integer $K$ large enough such that $r-2-K(r-1-r\delta/q)<-1$ and $-rK\alpha/q<-1$. Hence, from (\ref{3.11}), we obtain $J_2<\infty$. Similarly, we can get $J_3<\infty$. In order to estimate $J_1$, we first prove that
\[
\sup_{x\ge 1}\frac{1}{x^{1/p}}\max_{1\le k\le n}\left|\sum_{i=1}^{k}\ee X_{ni}^{(1)}\right|\rightarrow 0 \mbox{    as $n\rightarrow\infty$.}
\]
Observe that (\ref{3.1}), Lemma 4.5 in Zhang \cite{Zhang2016a}, and H\"{o}lder's inequality under sublinear expectations imply that $\ee|X|^{r/q}<\infty$ and $\ee|X|^{1/q}<\infty$. When $q>1$, observe that $|X_{ni}^{(1)}|\le x^{\alpha}n^{-\delta}$ and $|X_{ni}^{(1)}|\le |a_{ni}X_{i}|$, by H\"{o}lder's inequality, we see that
\begin{eqnarray}\label{3.12}
\nonumber \max_{1\le k\le n}\left|\sum_{i=1}^{k}\ee X_{ni}^{(1)}\right|&\le&\sum_{i=1}^{n}\ee\left|X_{ni}^{(1)}\right|\le x^{\alpha(1-1/q)}n^{-\delta(1-1/q)}\sum_{i=1}^{n}\ee\left|a_{ni}X_{i}\right|^{1/q}\\
\nonumber&\ll&x^{\alpha(1-1/q)}n^{-\delta(1-1/q)}\sum_{i=1}^{n}\left|a_{ni}\right|^{1/q}\\
\nonumber&\le&x^{\alpha(1-1/q)}n^{-\delta(1-1/q)}n^{(r-1)/r}\left(\sum_{i=1}^{n}|a_{ni}|^{r/q}\right)^{1/r}\\
&\approx&x^{\alpha(1-1/q)}n^{-\delta(1-1/q)}.
\end{eqnarray}
Observing that $\alpha(1-1/q)<\alpha<1/p$, by (\ref{3.12}), for any $x\ge 1$, we get
\[
\sup_{x\ge 1}\frac{1}{x^{1/p}}\max_{1\le k\le n}\left|\sum_{i=1}^{k}\ee X_{ni}^{(1)}\right|\ll n^{-\delta(1-1/q)}\rightarrow 0 \mbox{    as $n\rightarrow\infty$.}
\]
When $1/2<q\le 1$, noticing that $\ee(X)=\ee(-X)=0$, by choosing $\delta$ small sufficiently such that $-\delta(1-r/q)+1-r<0$, we see that
\begin{eqnarray*}
 \max_{1\le k\le n}\left|\sum_{i=1}^{k}\ee X_{ni}^{(1)}\right|&\le&2\sum_{i=1}^{n}\ee\left|a_{ni}X_{i}\right|I\left(\left|a_{ni}X_{i}\right|>x^{\alpha}n^{-\delta}\right)\\
&\le&2x^{\alpha(1-r/q)}n^{-\delta(1-r/q)}\sum_{i=1}^{n}\ee\left|a_{ni}X_i\right|^{r/q}\\
&\le&2x^{\alpha(1-r/q)}n^{-\delta(1-r/q)}\sum_{i=1}^{n}\left|a_{ni}\right|^{r/q}\approx x^{\alpha(1-r/q)}n^{-\delta(1-r/q)+1-r}.
\end{eqnarray*}
Observing that $1-r/q<0$, we have
 \[
\sup_{x\ge 1}\frac{1}{x^{1/p}}\max_{1\le k\le n}\left|\sum_{i=1}^{k}\ee X_{ni}^{(1)}\right|\ll n^{-\delta(1-r/q)+1-r}\rightarrow 0 \mbox{    as $n\rightarrow\infty$.}
\]
Thus, to establish $J_1<\infty$, we only need to prove that
\begin{equation}\label{3.13}
J_1^{*}:=\sum_{n=1}^{\infty}n^{r-2} \int_{1}^{\infty}\vv\left(\max_{1\le k\le n}\left|\sum_{i=1}^{k} (X_{ni}^{(1)}-\ee X_{ni}^{(1)})\right|>\frac{x^{1/p}}{8}\right)\dif x<\infty.
\end{equation}
Observe that $\{X_{ni}^{(1)}, 1\le i\le n, n\ge 1\}$ is independent, identically distributed under $(\Omega,\HH,\ee)$. By Markov's inequality under sublinear expectations, $C_r$'s inequality and Lemma \ref{lem1}, we conclude that for a suitably large $M$, which will be determined later,
\begin{eqnarray}\label{3.14}
\nonumber &&\vv\left(\max_{1\le k\le n}\left|\sum_{i=1}^{k}( X_{ni}^{(1)}-\ee X_{ni}^{(1)})\right|>\frac{x^{1/p}}{8}\right)\\
&&\ll x^{-M/p}(\ln n)^M\left\{\sum_{i=1}^{n}\ee|X_{ni}^{(1)}|^M+\left[\sum_{i=1}^{n}\ee|X_{ni}^{(1)}|^2\right]^{M/2}\right\}.
\end{eqnarray}
Choosing large sufficiently $M$ such that $-M(1/p-\alpha)-r\alpha/q<-1$, $-1-(M-r/q)\delta<-1$, we have
\begin{eqnarray}\label{3.15}
\nonumber &&\sum_{n=1}^{\infty}n^{r-2}(\ln n)^M\sum_{i=1}^{n}\int_{1}^{\infty} x^{-M/p}\ee|X_{ni}^{(1)}|^M\dif x\\
\nonumber&&\ll \sum_{n=1}^{\infty}n^{r-2}n^{-\delta (M-r/q)}(\ln n)^M\sum_{i=1}^{n}|a_{ni}|^{r/q}\int_{1}^{\infty}x^{-M(1/p-\alpha)-r\alpha/q}\dif x\\
&&\ll \sum_{n=1}^{\infty}n^{-1-\delta (M-r/q)}(\ln n)^M<\infty.
\end{eqnarray}
When $r/q\ge 2$, (\ref{3.1}) implies that $\ee|X|^2<\infty$. Observing that $q>1/2$, we could choose $M$ large enough such that $-M/p<-1$, $r-2-(2q-1)M/2<-1$. Then, by $\beta>-q/r\ge -\frac12$,
\begin{eqnarray}\label{3.16}
\nonumber &&\sum_{n=1}^{\infty}n^{r-2}(\ln n)^M\int_{1}^{\infty} x^{-M/p}\left[\sum_{i=1}^{n}\ee|X_{ni}^{(1)}|^2\right]^{M/2}\dif x\\
\nonumber&&\ll \sum_{n=1}^{\infty}n^{r-2}(\ln n)^M\left(\sum_{i=1}^{n}a_{ni}^{2}\right)^{M/2}\int_{1}^{\infty} x^{-M/p}\dif x\\
\nonumber&&\ll \sum_{n=1}^{\infty}n^{r-2}(\ln n)^M\left(\sum_{i=1}^{n}i^{2\beta}n^{-2(q+\beta)}\right)^{M/2}\\
&&\ll\sum_{n=1}^{\infty}n^{r-2-(2q-1)M/2}(\ln n)^M<\infty.
\end{eqnarray}
When $r/q<2$, choosing $M$ large enough such that $-M[1/p-\alpha+\alpha r/(2q)]<-1$, $r-2-[\delta(2-r/q)+r-1]M/2<-1$, we have
\begin{eqnarray}\label{3.17}
\nonumber &&\sum_{n=1}^{\infty}n^{r-2}(\ln n)^M\int_{1}^{\infty} x^{-M/p}\left[\sum_{i=1}^{n}\ee|X_{ni}^{(1)}|^2\right]^{M/2}\dif x\\
\nonumber&&\ll \sum_{n=1}^{\infty}n^{r-2}n^{-\delta(2-r/q)M/2}(\ln n)^M\left(\sum_{i=1}^{n}a_{ni}^{r/q}\right)^{M/2}\int_{1}^{\infty} x^{-M[1/p-\alpha+r\alpha/(2q)]}\dif x\\
&&\ll\sum_{n=1}^{\infty}n^{r-2-[\delta(2-r/q)+r-1]M/2}(\ln n)^M<\infty.
\end{eqnarray}
Consequently, by (\ref{3.14}), (\ref{3.15}), (\ref{3.16}) and (\ref{3.17}), we get $J_1^{*}<\infty$.
Now, we prove that (\ref{3.2}) implies (\ref{3.1}). By Markov's inequality under sublinear expectations (see (8) of Wu \cite{Wu2020}), (\ref{3.2}), and Lemma 4.5 (\rmnum 3) in Zhang \cite{Zhang2016a},
\begin{eqnarray}\label{3.18}
\nonumber &&\quad\sum_{n=1}^{\infty}n^{r-2}\vv\left(\max_{1\le k\le n}\left|\sum_{i=1}^{k}a_{ni}X_i\right|>\epsilon\right)\\
\nonumber &&=\sum_{n=1}^{\infty}n^{r-2}\ee\left(1\cdot I\left(\max_{1\le k\le n}\left|\sum_{i=1}^{k}a_{ni}X_i\right|>\epsilon\right)\right)\\
\nonumber &&\le \sum_{n=1}^{\infty}n^{r-2}\ee\left(\left(\left(\max_{1\le k\le n}\left|\sum_{i=1}^{k}a_{ni}X_i\right|^p-(\epsilon/2)^p\right)^{+}/(\epsilon/2)^p\right)I\left(\max_{1\le k\le n}\left|\sum_{i=1}^{k}a_{ni}X_i\right|>\epsilon\right)\right)\\
\nonumber &&\le \sum_{n=1}^{\infty}n^{r-2}\ee\left(\left(\left(\max_{1\le k\le n}\left|\sum_{i=1}^{k}a_{ni}X_i\right|^p-(\epsilon/2)^p\right)^{+}/(\epsilon/2)^p\right)\right)\\
&&\le\sum_{n=1}^{\infty}n^{r-2}\ee\left(\max_{1\le k\le n}\left|\sum_{i=1}^{k}a_{ni}X_i\right|^p-(\epsilon/2)^p\right)^{+}/(\epsilon/2)^p<\infty.
\end{eqnarray}
Since
\[
\max_{1\le k\le n}\left|a_{nk}X_k\right|\le 2\max_{1\le k\le n}\left|\sum_{i=1}^{k}a_{ni}X_i\right|,
\]
by (\ref{3.18}) and similar proofs of (3.17) in Guo \cite{Guo2012}, we conclude that
\begin{eqnarray}\label{3.19}
\vv\left(\max_{1\le k\le n}\left|a_{nk}X_k\right|>\epsilon\right)\rightarrow 0  \mbox{   as $n\rightarrow \infty$.}
\end{eqnarray}
Therefore, by (\ref{3.19}) and Lemma \ref{lem2}, we obtain
\begin{eqnarray}\label{3.20}
\sum_{i=1}^{n}\vv\left(\left|a_{ni}X_i\right|>\epsilon\right)\ll \vv\left(\max_{1\le k\le n}\left|a_{nk}X_k\right|>\epsilon\right) \mbox{   for all $\epsilon>0$.}
\end{eqnarray}
Now, combining (\ref{3.20}) with (\ref{3.2}) gives
\begin{eqnarray}\label{3.21}
\sum_{n=1}^{\infty}n^{r-2}\int_{\epsilon}^{\infty}\sum_{i=1}^{n}\vv\left(\left|a_{ni}X_i\right|>x^{1/p}\right)\dif x<\infty.
\end{eqnarray}
From the proof of (\ref{3.9}), we conclude that (\ref{3.21}) is equivalent to (\ref{3.1}).
\end{proof}
\begin{proof}[Proof of Theorem \ref{thm2}] As in the proof of Theorem \ref{thm1}, with Lemma \ref{lem4} in place of Lemma \ref{lem3}, and using
\[
\sum_{i=1}^{n}a_{ni}^{r/q}\approx \sum_{i=1}^{n}n^{-r(q+\beta)/q}i^{\beta r/q}\approx n^{-r(q+\beta)/q}\ln n,\\
\int_{1}^{s^{1/q}}t^{\beta(r-1)/(q+\beta)}\dif t=\int_{1}^{s^{1/q}}t^{-1}\dif t\approx \ln(s)
\]
when $\beta=-q/r$, we could prove Theorem \ref{thm2}. Hence the proof is omitted.
\end{proof}
\begin{proof}[Proof of Theorem \ref{thm3}] As in the proof of Theorem \ref{thm1}, using
\[
\sum_{i=1}^{n}a_{ni}^{r/q}\approx \sum_{i=1}^{n}n^{-r(q+\beta)/q}i^{\beta r/q}\approx n^{-r(q+\beta)/q} \mbox{ if $\beta<-q/r$}
\]
and with
\[
\int_{1}^{s^{1/q}}t^{\beta(r-1)/(q+\beta)}\dif t\approx C \mbox{ if $\beta<-q/r$}
\]
in place of (\ref{3.8}), we could prove Theorem \ref{thm3}. Thus the proof is omitted.
\end{proof}
\begin{proof}[Proof of Theorem \ref{thm4}] Setting $a_{ni}=(A_{n-i}^{\alpha-1}/A_n^{\alpha})^q$, $0\le i\le n$, $n\ge 1$, observe that $a_{ni}\approx (n-i)^{q(\alpha-1)}n^{-q\alpha}$, $0\le i<n$, $n\ge 1$, $a_{nn}\approx n^{-q\alpha}$. As in the proof of Theorem 14 in Guo and Shan \cite{Guo2020}, letting $\beta=q(\alpha-1)$ in Corollary \ref{cor}, we finish the proof of Theorem \ref{thm4}.
\end{proof}

{\bf Acknowledgements}

The authors are very grateful to the referees and reviewers, whose careful reading and beneficial comments helped to correct many misprints, mistakes and improve the quality of the paper.

{\bf Funding}

This research was supported by Doctoral Scientific Research Starting Foundation of Jingdezhen Ceramic University ( Nos.102/01003002031 ), Scientific Program of Department of Education of Jiangxi Province of China (Nos. GJJ190732, GJJ180737), Natural Science Foundation Program of Jiangxi Province 20202BABL211005, and National Natural Science Foundation of China (Nos. 61662037).

{\bf Availability of data and materials}

No data were used to support this study.

{\bf Competing interests}

The authors declare that they have no competing interests.

{\bf Authors¡¯ contributions}

All authors contributed equally and read and approved the final manuscript.

%\end{CJK*}
\end{document}